\documentclass[a4paper]{amsart}
\author[Raghavan]{Dilip Raghavan}
\thanks{The author was partially supported by the Singapore Ministry of Education's research grant number A-8001467-00-00.}
\address[Raghavan]{Department of Mathematics\\
National University of Singapore\\
Singapore 119076.}
\email{\href{dilip.raghavan@protonmail.com}{dilip.raghavan@protonmail.com}}
\urladdr{\url{https://dilip-raghavan.github.io/}}
\date{\today}
\title{Q-points in the Tukey order}
\usepackage[only, llbracket, rrbracket]{stmaryrd}
\usepackage{amssymb, amsmath, amsthm, mathrsfs, enumitem, amsfonts, latexsym, bbm}
\usepackage{hyperref}
\def\polhk#1{\setbox0=\hbox{#1}{\ooalign{\hidewidth
    \lower1.5ex\hbox{`}\hidewidth\crcr\unhbox0}}}
\newtheorem{Theorem}{Theorem}[section]

\newtheorem{Lemma}[Theorem]{Lemma}
\newtheorem{Cor}[Theorem]{Corollary}

\theoremstyle{definition}
\newtheorem{Def}[Theorem]{Definition}

\theoremstyle{remark}

\newcommand{\cc}{\mathfrak{c}}

\newcommand{\dd}{{\mathfrak{d}}}

\renewcommand{\[}{\left[}
\renewcommand{\]}{\right]}

\newcommand{\lc}{\left|}
\newcommand{\rc}{\right|}

\DeclareMathOperator{\cov}{cov}

\newcommand{\MMM}{\mathcal{M}}

\newcommand{\UUU}{{\mathcal{U}}}
\newcommand{\VVV}{{\mathcal{V}}}

\newcommand{\FFF}{{\mathcal{F}}}

\newcommand{\seq}[4]{\left\langle {#1}_{#2}: #2 #3 #4 \right\rangle}

\newcommand{\pc}[2]{{\[#1\]}^{#2}}

\begin{document}
\begin{abstract}
 Q-points are cofinal in the RK-ordering under several mild hypotheses.
\end{abstract}
\maketitle
\section{Introduction} \label{sec:intro}
The purpose of this note is to address a question that has been considered in two recent papers, namely the existence of Tukey maximal Q-points.
This was asked in~\cite{arXiv:2304.07214} and addressed in \cite{arXiv:2404.02379}, where it was proved by a forcing argument that Tukey maximal Q-points may consistently exist.
We show here that fairly standard constructions using mild hypotheses yield many Tukey maximal Q-points.
If $\dd = {\aleph}_{1}$ or if there are infinitely many pairwise RK-non-isomorphic selective ultrafilters, then the Q-points are cofinal in the RK ordering.
In particular, there are ${2}^{\cc}$ pairwise RK-incomparable Tukey maximal Q-points.

For general facts about the Tukey theory of directed sets we refer the reader to \cite{cofinal} or \cite{suslincombdic}.
For the specific case of ultrafilters on $\omega$, see, for example, \cite{tukey} or \cite{nextbestpaper}.
The order structure of P-points in the RK and Tukey ordering is considered in \cite{pomegaembed}, \cite{longchain}, \cite{pchains}, \cite{lowerbounds}, among other places.
Consistency results can be found in \cite{arXiv:2408.04446}, \cite{smallu}. 
Our notation is standard.
We refer to \cite{sanu-survey} for any undefined terms.
\section{Q-points are cofinal}
\begin{Def} \label{def:ip}
 An \emph{interval partition or IP} is a sequence $I = \seq{i}{n}{\in}{\omega} \in {\omega}^{\omega}$ such that ${i}_{0} = 0$ and $\forall n \in \omega\[{i}_{n} < {i}_{n+1}\]$.
 
 Given an IP $I$ and $n \in \omega$, ${I}_{n}$ denotes $\[ {i}_{n}, {i}_{n+1} \right) = \{l \in \omega: {i}_{n} \leq l < {i}_{n+1} \}$.
\end{Def}
Recall that an ultrafilter $\UUU$ on $\omega$ is a \emph{Q-point} if for every finite-to-one function $f: \omega \rightarrow \omega$, there exists $A \in \UUU$ such that $f$ is one-to-one on $A$.
It is well known (see e.g.\@ \cite{BJ}) that an ultrafilter $\UUU$ on $\omega$ is a Q-point if and only if for every IP $I$, there exists $A \in \UUU$ such that $\forall k \in \omega\[\lc {I}_{k} \cap A \rc \leq 1\]$.

Recall that $\dd$ is the minimal cardinality of a dominating family of functions in ${\omega}^{\omega}$.
It is well-known (see e.g.\@ \cite{blasssmall}) that $\dd$ is the minimal $\kappa$ such that there exists a sequence $\seq{I}{\alpha}{<}{\kappa}$ of IPs such that for every IP $I$, there exists $\alpha < \kappa$ such that $\forall l \in \omega \exists k \in \omega\[{I}_{k} \subseteq {I}^{\alpha}_{l}\]$.
This is the characterization of $\dd$ which is useful below.
Recall that Canjar~\cite{MR0993747} showed that if $\dd = {\aleph}_{1}$, then any filter on $\omega$ that is generated by fewer than ${2}^{{\aleph}_{0}}$ sets can be extended to a Q-point.
\begin{Lemma} \label{lem:choice}
 Suppose $\seq{Y}{n}{\in}{\omega}$ is a sequence such that:
 \begin{enumerate}[series=choice]
  \item
  ${Y}_{n} \in \pc{\omega}{{\aleph}_{0}}$, for all $n \in \omega$;
  \item
  ${Y}_{n} \cap {Y}_{m} = \emptyset$, for all $n < m < \omega$;
 \end{enumerate}
 Suppose $I = \seq{i}{n}{\in}{\omega}$ is an IP.
 Then there exists $X \in \pc{\omega}{{\aleph}_{0}}$ such that:
 \begin{enumerate}[resume=choice]
  \item
  $\lc X \cap {Y}_{n} \rc = {\aleph}_{0}$, for all $n \in \omega$;
  \item
  $\lc X \cap {I}_{l} \rc \leq 1$, for all $l \in \omega$.
 \end{enumerate}
\end{Lemma}
\begin{proof}
 Define ${Z}_{n} = \{l \in \omega: {I}_{l} \cap {Y}_{n} \neq \emptyset\}$, for all $n \in \omega$.
 ${Z}_{n} \in \pc{\omega}{{\aleph}_{0}}$ because ${Y}_{n} \in \pc{\omega}{{\aleph}_{0}}$ and because $I$ is an IP.
 Let $\seq{T}{n}{\in}{\omega}$ be such that:
 \begin{enumerate}[resume=choice]
  \item
  ${T}_{n} \in \pc{{Z}_{n}}{{\aleph}_{0}}$, for all $n \in \omega$;
  \item
  ${T}_{m} \cap {T}_{n} = \emptyset$, for all $m < n < \omega$.
 \end{enumerate}
 For each $l \in {T}_{n}$, since ${I}_{l} \cap {Y}_{n} \neq \emptyset$, choose ${k}_{l, n} \in {I}_{l} \cap {Y}_{n}$.
 Let
 \begin{align*}
  X = \{{k}_{l, n}: l \in {T}_{n} \wedge n \in \omega\}.
 \end{align*}
 For each $n \in \omega$, $\{{k}_{l, n}: l \in {T}_{n}\} \subseteq X \cap {Y}_{n}$, and if $l \neq l'$, then ${k}_{l, n} \neq {k}_{l', n}$ because ${k}_{l, n} \in {I}_{l}$, ${k}_{l', n} \in {I}_{l'}$, and ${I}_{l} \cap {I}_{l'} = \emptyset$.
 Therefore, (3) is satisfied.
 For (4), suppose that $n, n' \in \omega$, $l \in {T}_{n}$, $l' \in {T}_{n'}$, and that ${k}_{l, n}, {k}_{l', n'} \in {I}_{j}$, for some $j \in \omega$.
 Then since ${k}_{l, n} \in {I}_{j} \cap {I}_{l}$ and ${k}_{l', n'} \in {I}_{j} \cap {I}_{l'}$, $l = j = l'$.
 Since $l \in {T}_{n} \cap {T}_{n'}$, $n = n'$, whence ${k}_{l, n} = {k}_{l', n'}$.
 Hence (4) is satisfied and $X$ is as needed.
\end{proof}
\begin{Lemma} \label{lem:induction}
 Suppose $\langle {I}^{\alpha}: \alpha < {\omega}_{1} \rangle$ is a sequence of IPs.
 Suppose also that $\seq{X}{n}{\in}{\omega}$ is a sequence such that:
 \begin{enumerate}[series=induction]
  \item
  ${X}_{n} \in \pc{\omega}{{\aleph}_{0}}$, for all $n \in \omega$;
  \item
  ${X}_{m} \cap {X}_{n} = \emptyset$, for all $m < n < \omega$;
  \item
  $\omega = {\bigcup}_{n \in \omega}{{X}_{n}}$.
 \end{enumerate}
 There exists a filter $\FFF$ on $\omega$ such that:
 \begin{enumerate}[resume=induction]
  \item
  $\forall Z \in \FFF \forall n \in \omega\[\lc Z \cap {X}_{n} \rc = {\aleph}_{0}\]$;
  \item
  $\forall \alpha < {\omega}_{1} \exists Z \in \FFF\forall l \in \omega\[\lc {I}^{\alpha}_{l} \cap Z \rc \leq 1\]$.
 \end{enumerate}
\end{Lemma}
\begin{proof}
 For any $Z \subseteq \omega$, let $Z\[n\]$ denote $Z \cap {X}_{n}$.
 Build by induction a family $\{{Z}_{\alpha}: \alpha < {\omega}_{1}\} \subseteq \pc{\omega}{{\aleph}_{0}}$ satisfying the following:
 \begin{enumerate}[resume=induction]
  \item
  $\forall \alpha < {\omega}_{1} \forall n \in \omega\[\lc {Z}_{\alpha}\[n\] \rc = {\aleph}_{0}\]$;
  \item
  $\forall \alpha < \beta < {\omega}_{1} \forall n \in \omega\[{Z}_{\beta}\[n\] \; {\subseteq}^{\ast} \; {Z}_{\alpha}\[n\]\]$;
  \item
  $\forall \alpha < {\omega}_{1}\forall l \in \omega\[\lc {I}^{\alpha}_{l} \cap {Z}_{\alpha} \rc \leq 1\]$.
 \end{enumerate}
 Suppose for a moment that this has been accomplished.
 Let
 \begin{align*}
  \FFF = \left\{ Z \subseteq \omega: \exists \alpha < {\omega}_{1}\forall n \in \omega\[{Z}_{\alpha}\[n\] \; {\subseteq}^{\ast} \; Z\[n\]\] \right\}.
 \end{align*}
 Then it is clear that $\FFF$ is a filter on $\omega$.
 If $Z \in \FFF$ as witnessed by $\alpha < {\omega}_{1}$, then for each $n \in \omega$, ${Z}_{\alpha}\[n\] \; {\subseteq}^{\ast} \; Z\[n\]$, so by (6), (4) holds.
 Since for each $\alpha < {\omega}_{1}$, ${Z}_{\alpha} \in \FFF$, (8) gives (5).
 Hence $\FFF$ is a required.
 
 To build $\{{Z}_{\alpha}: \alpha < {\omega}_{1}\}$, proceed by induction.
 Fix $\alpha < {\omega}_{1}$ and suppose $\{{Z}_{\xi}: \xi < \alpha\} \subseteq \pc{\omega}{{\aleph}_{0}}$ satisfying (6)--(8) is given.
 For each $n \in \omega$, the family $\{{Z}_{\xi}\[n\]: \xi < \alpha\} \subseteq \pc{{X}_{n}}{{\aleph}_{0}}$ satisfies $\forall \zeta < \xi < \alpha\[{Z}_{\xi}\[n\] \; {\subseteq}^{\ast} \; {Z}_{\zeta}\[n\]\]$.
 As $\alpha$ is countable, find ${Y}_{n} \in \pc{{X}_{n}}{{\aleph}_{0}}$ such that $\forall \xi < \alpha\[{Y}_{n} \; {\subseteq}^{\ast} \; {Z}_{\xi}\[n\]\]$.
 Note that ${Y}_{m} \cap {Y}_{n} = \emptyset$ for $n \neq m$ because ${X}_{m} \cap {X}_{n} = \emptyset$.
 Applying Lemma \ref{lem:choice}, find $Z \in \pc{\omega}{{\aleph}_{0}}$ satisfying (3) and (4) of Lemma \ref{lem:choice} for ${I}^{\alpha}$.
 Define ${Z}_{\alpha} = Z \cap \left( {\bigcup}_{m \in \omega}{{Y}_{m}} \right)$.
 It is clear that (8) follows from (4) of Lemma \ref{lem:choice}.
 Note that for any $n \in \omega$, ${Z}_{\alpha}\[n\] = Z \cap {Y}_{n}$.
 Therefore, (6) follows from (3) of Lemma \ref{lem:choice} and (7) follows from the choice of ${Y}_{n}$.
 Thus ${Z}_{\alpha}$ has all the required properties.
 This concludes the construction and the proof.
\end{proof}
\begin{Theorem} \label{thm:q-cofinal}
 Assume $\dd = {\aleph}_{1}$.
 Let $\UUU$ be any ultrafilter on $\omega$ and let $f: \omega \rightarrow \omega$ be such that $\lc {f}^{-1}(\{n\}) \rc = {\aleph}_{0}$, for all $n \in \omega$.
 There there exists $\VVV$ such that:
 \begin{enumerate}
  \item
  $\VVV$ is a Q-point;
  \item
  $f$ witnesses that $\UUU \; {\leq}_{RK} \; \VVV$.
 \end{enumerate}
\end{Theorem}
\begin{proof}
 Fix a sequence $\langle {I}^{\alpha}: \alpha < {\omega}_{1} \rangle$ of IPs so that for every IP $I$, there exists $\alpha < {\omega}_{1}$ such that $\forall l \in \omega \exists k \in \omega\[{I}_{k} \subseteq {I}^{\alpha}_{l}\]$.
 Define ${X}_{n} = {f}^{-1}(\{n\})$, for every $n \in \omega$.
 Applying Lemma \ref{lem:induction}, find a filter $\FFF$ on $\omega$ satisfying (4) and (5) of Lemma \ref{lem:induction}.
 Suppose $Z \in \FFF$ and $A \in \UUU$.
 Then ${f}^{-1}(A) = {\bigcup}_{n \in A}{{X}_{n}}$, whence $Z \cap {f}^{-1}(A) = {\bigcup}_{n \in A}{\left(Z \cap {X}_{n}\right)}$, which is infinite by (4) of Lemma \ref{lem:induction}.
 It follows that there exists an ultrafilter $\VVV$ on $\omega$ such that $\FFF \subseteq \VVV$ and $\{{f}^{-1}(A): A \in \UUU\} \subseteq \VVV$.
 Now (2) is immediate by the choice of $\VVV$.
 To see (1), let $I$ be any IP.
 Let $\alpha < {\omega}_{1}$ be such that $\forall l \in \omega \exists k \in \omega\[{I}_{k} \subseteq {I}^{\alpha}_{l}\]$.
 By (5) of Lemma \ref{lem:induction}, there exists $Z \in \FFF \subseteq \VVV$ such that $\forall l \in \omega\[\lc {I}^{\alpha}_{l} \cap Z \rc \leq 1\]$.
 It is easily seen that $\forall n \in \omega\[\lc {I}_{n} \cap Z \rc \leq 2\]$.
 Define $R = \{\min\left( {I}_{n} \cap Z \right): n \in \omega \wedge {I}_{n} \cap Z \neq \emptyset\}$ and $S = Z \setminus R$.
 It is clear that $\forall n \in \omega\[\lc {I}_{n} \cap R \rc \leq 1\]$ and that $\forall n \in \omega\[\lc {I}_{n} \cap S \rc \leq 1\]$.
 $R \in \VVV$ or $S \in \VVV$ because $Z \in \VVV$ and $\VVV$ is an ultrafilter.
 Therefore, $\VVV$ is a Q-point (see \cite[Lemma 7.1]{souvssu} for a similar argument).
\end{proof}
\begin{Cor} \label{cor:manyq-maximal}
 Assume $\dd = {\aleph}_{1}$.
 There exist $\seq{\VVV}{\alpha}{<}{{2}^{{2}^{{\aleph}_{0}}}}$ such that:
 \begin{enumerate}[series=manyq-maximal]
  \item
  ${\VVV}_{\alpha}$ is a Q-point for every $\alpha < {2}^{{2}^{{\aleph}_{0}}}$;
  \item
  ${\VVV}_{\alpha} \; {\not\leq}_{RK} \; {\VVV}_{\beta}$, for $\alpha, \beta < {2}^{{2}^{{\aleph}_{0}}}$ with $\alpha \neq \beta$;
  \item
  ${\VVV}_{\alpha}$ is Tukey maximal, for all $\alpha < {2}^{{2}^{{\aleph}_{0}}}$.
 \end{enumerate}
\end{Cor}
\begin{proof}
 There exists a family of ultrafilters $\{{\UUU}_{\alpha}: \alpha < {2}^{{2}^{{\aleph}_{0}}}\}$ on $\omega$ such that each ${\UUU}_{\alpha}$ is Tukey maximal and ${\UUU}_{\alpha} \neq {\UUU}_{\beta}$, for $\alpha \neq \beta$.
 Fix $f: \omega \rightarrow \omega$ such that $\lc {f}^{-1}(\{n\}) \rc = {\aleph}_{0}$.
 By Theorem \ref{thm:q-cofinal}, find ${\VVV}_{\alpha}$ satisfying (1) and (2) of Theorem \ref{thm:q-cofinal} for $f$ and ${\UUU}_{\alpha}$.
 Each ${\VVV}_{\alpha}$ is a Q-point which is Tukey maximal as it is RK-above ${\UUU}_{\alpha}$.
 Since ${\UUU}_{\alpha} = \{A \subseteq \omega: {f}^{-1}(A) \in {\VVV}_{\alpha}\}$, it follows that ${\VVV}_{\alpha} \neq {\VVV}_{\beta}$, whenever $\alpha \neq \beta$.
 For each $\alpha < {2}^{{2}^{{\aleph}_{0}}}$, let $F(\alpha) = \{\beta < {2}^{{2}^{{\aleph}_{0}}}: {\VVV}_{\beta} \; {\leq}_{RK} \; {\VVV}_{\alpha}\}$.
 Since $\lc F(\alpha) \rc \leq {2}^{{\aleph}_{0}}$, by a fundamental theorem on set mappings (see e.g.\@ \cite{partitionbible}), there exists $X \subseteq {2}^{{2}^{{\aleph}_{0}}}$ with $\lc X \rc = {2}^{{2}^{{\aleph}_{0}}}$ such that for each $\alpha \in X$, $F(\alpha) \cap X = \{\alpha\}$.
 Therefore, $\seq{\VVV}{\alpha}{\in}{X}$ is as needed.
\end{proof}
It is easy to see that the construction in Theorem \ref{thm:q-cofinal} can also be carried out if there are infinitely many pairwise RK-non-isomorphic selective ultrafilters.
It is well-known that this happens if, for example, $\cov(\MMM) = \cc$.
We refer the reader to \cite{bulletin} for a discussion of the differences between constructions from hypotheses of the form $\dd = {\aleph}_{1}$ and those from hypotheses of the form $\cov(\MMM) = \cc$.
We give a few details below.
\begin{Lemma} \label{lem:noniso}
 Suppose $\UUU$ and $\VVV$ are selective ultrafilters on $\omega$ with $\UUU \; {\not\equiv}_{RK} \; \VVV$.
 For any IP $I$, there exist $X \in \UUU$ and $Y \in \VVV$ such that:
 \begin{enumerate}
  \item
  $\forall n \in \omega\[\lc {I}_{n} \cap X \rc \leq 1\]$ and $\forall n \in \omega\[\lc {I}_{n} \cap Y \rc \leq 1\]$;
  \item
  $\left\{n \in \omega: {I}_{n} \cap X \neq \emptyset \right\} \cap \left\{n \in \omega: {I}_{n} \cap Y \neq \emptyset \right\} = \emptyset$.
 \end{enumerate}
\end{Lemma}
\begin{proof}
 Find $A \in \UUU$ and $B \in \VVV$ such that $\forall n \in \omega\[\lc {I}_{n} \cap A \rc \leq 1\]$ and
 \begin{align*}
  \forall n \in \omega\[\lc {I}_{n} \cap B \rc \leq 1\].
 \end{align*}
 Let $f: \omega \rightarrow \omega$ be the function such that $f''{I}_{n} = \{n\}$, for all $n \in \omega$.
 Let ${\UUU}^{\ast} = \{U \subseteq \omega: {f}^{-1}(U) \in \UUU\}$ and let ${\VVV}^{\ast} = \{V \subseteq \omega: {f}^{-1}(V) \in \VVV\}$.
 Then ${\UUU}^{\ast} \neq {\VVV}^{\ast}$ because ${\UUU}^{\ast} \; {\equiv}_{RK} \; \UUU \; {\not\equiv}_{RK} \; \VVV \; {\equiv}_{RK} \; {\VVV}^{\ast}$.
 So there exist $U \in {\UUU}^{\ast}$ and $V \in {\VVV}^{\ast}$ with $U \cap V = \emptyset$.
 Then $X = A \cap {f}^{-1}(U) \in \UUU$ and $Y = B \cap {f}^{-1}(V) \in \VVV$ are as required.
\end{proof}
\begin{Cor} \label{cor:selectives}
 Assume there exists a family $\{{\UUU}_{n}:  n \in \omega\}$ such that:
 \begin{enumerate}[series=selectives]
  \item
  ${\UUU}_{n}$ is a selective ultrafilter on $\omega$;
  \item
  for each $n, m \in \omega$, if $n \neq m$, then ${\UUU}_{m} \; {\not\equiv}_{RK} \; {\UUU}_{n}$.
 \end{enumerate}
 Let $\UUU$ be any ultrafilter on $\omega$ and let $f: \omega \rightarrow \omega$ be such that $\lc {f}^{-1}(\{n\}) \rc = {\aleph}_{0}$, for all $n \in \omega$.
 There there exists $\VVV$ such that:
 \begin{enumerate}[resume=selectives]
  \item
  $\VVV$ is a Q-point;
  \item
  $f$ witnesses that $\UUU \; {\leq}_{RK} \; \VVV$.
 \end{enumerate}
\end{Cor}
\begin{proof}
 Define ${X}_{n} = {f}^{-1}(\{n\})$.
 Let ${\VVV}_{n}$ be an RK-isomorphic copy of ${\UUU}_{n}$ with ${X}_{n} \in {\VVV}_{n}$, and define
 \begin{align*}
  \VVV = \left\{A \subseteq \omega: \left\{n \in \omega: A \cap {X}_{n} \in {\VVV}_{n} \right\} \in \UUU\right\}.
 \end{align*}
 $\VVV$ is as required.
 To see this, suppose $I$ is any IP.
 Using Lemma \ref{lem:noniso}, fix ${A}_{m, n} \in {\VVV}_{m}$ and ${B}_{m, n} \in {\VVV}_{n}$ satisfying (1) and (2) of Lemma \ref{lem:noniso}, for all $m < n < \omega$.
 For each $m \in \omega$, find ${C}_{m} \in {\VVV}_{m}$ with ${C}_{m} \subseteq {A}_{m, m+1}$ and $\forall m < n < \omega\[{C}_{m} \; {\subseteq}^{\ast} \; {A}_{m, n}\]$.
 Note $\forall k \in \omega\[\lc {I}_{k} \cap {C}_{m} \rc \leq 1\]$.
 For each $n \in \omega$, define ${D}_{n} = {C}_{n} \cap \left( {\bigcap}_{m < n}{{B}_{m, n}} \right) \in {\VVV}_{n}$, with $\bigcap \emptyset$ taken to be $\omega$.
 For each $n \in \omega$, find ${k}_{n} \in \omega$ such that $\forall m < n\[{D}_{m} \setminus {A}_{m, n} \subseteq {\bigcup}_{l \leq {k}_{n}}{{I}_{l}}\]$.
 Define ${E}_{n} = {D}_{n} \setminus \left( {\bigcup}_{l \leq {k}_{n}}{{I}_{l}} \right) \in {\VVV}_{n}$ and define $E = {\bigcup}_{n \in \omega}{\left( {E}_{n} \cap {X}_{n} \right)}$.
 It is easily seen that $E \in \VVV$ and that $\forall l \in \omega\[\lc {I}_{l} \cap E \rc \leq 1\]$.
 Therefore, (3) is satisfied and it is clear from the definition of $\VVV$ that (4) holds as well (see \cite{sanu-survey} for a similar argument).
\end{proof}
\def\polhk#1{\setbox0=\hbox{#1}{\ooalign{\hidewidth
  \lower1.5ex\hbox{`}\hidewidth\crcr\unhbox0}}}
\providecommand{\bysame}{\leavevmode\hbox to3em{\hrulefill}\thinspace}
\providecommand{\MR}{\relax\ifhmode\unskip\space\fi MR }
\providecommand{\MRhref}[2]{%
  \href{http://www.ams.org/mathscinet-getitem?mr=#1}{#2}
}
\providecommand{\href}[2]{#2}


\begin{thebibliography}{10}

\bibitem{BJ}
T.~Bartoszy{\'n}ski and H.~Judah, \emph{Set theory: On the structure of the
  real line}, A K Peters Ltd., Wellesley, MA, 1995.

\bibitem{arXiv:2304.07214}
T.~Benhamou and N.~Dobrinen, \emph{Cofinal types of ultrafilters over
  measurable cardinals}, arXiv:2304.07214, preprint (2023), 34 pages.

\bibitem{arXiv:2404.02379}
T.~Benhamou and F.~Wu, \emph{Diamond principles and {T}ukey-top ultrafilters on
  a countable set}, arXiv:2404.02379, preprint (2024), 30 pages.

\bibitem{blasssmall}
A.~Blass, \emph{Combinatorial cardinal characteristics of the continuum},
  Handbook of set theory. {V}ols. 1, 2, 3, Springer, Dordrecht, 2010,
  pp.~395--489.

\bibitem{nextbestpaper}
A.~Blass, N.~Dobrinen, and D.~Raghavan, \emph{The next best thing to a
  {P}-point}, J. Symb. Log. \textbf{80} (2015), no.~3, 866--900.

\bibitem{MR0993747}
R.~M. Canjar, \emph{On the generic existence of special ultrafilters}, Proc.
  Amer. Math. Soc. \textbf{110} (1990), no.~1, 233--241.

\bibitem{partitionbible}
P.~Erd{\H{o}}s, A.~Hajnal, A.~M{\'a}t{\'e}, and R.~Rado, \emph{Combinatorial
  set theory: partition relations for cardinals}, Studies in Logic and the
  Foundations of Mathematics, vol. 106, North-Holland Publishing Co.,
  Amsterdam, 1984.

\bibitem{longchain}
B.~Kuzeljevic and D.~Raghavan, \emph{A long chain of {P}-points}, J. Math. Log.
  \textbf{18} (2018), no.~1, 1850004, 38.

\bibitem{lowerbounds}
B.~Kuzeljevic, D.~Raghavan, and J.~L. Verner, \emph{Lower bounds of sets of
  {P}-points}, Notre Dame J. Form. Log. \textbf{64} (2023), no.~3, 317--327.

\bibitem{sanu-survey}
Borisa Kuzeljevic and Dilip Raghavan, \emph{Order structure of {P}-point
  ultrafilters and their relatives}, arXiv:2404.03238, submitted (2024), 29 pp.

\bibitem{bulletin}
D.~Raghavan, \emph{Almost disjoint families and diagonalizations of length
  continuum}, Bull. Symbolic Logic \textbf{16} (2010), no.~2, 240--260.

\bibitem{pomegaembed}
D.~Raghavan and S.~Shelah, \emph{On embedding certain partial orders into the
  {P}-points under {R}udin--{K}eisler and {T}ukey reducibility}, Trans. Amer.
  Math. Soc. \textbf{369} (2017), no.~6, 4433--4455.

\bibitem{souvssu}
D.~Raghavan and J.~Stepr{\=a}ns, \emph{Stable ordered-union versus selective
  ultrafilters}, arXiv:2302.05539, submitted (2023), 43 pp.

\bibitem{tukey}
D.~Raghavan and S.~Todorcevic, \emph{Cofinal types of ultrafilters}, Ann. Pure
  Appl. Logic \textbf{163} (2012), no.~3, 185--199.

\bibitem{pchains}
D.~Raghavan and J.~L. Verner, \emph{Chains of {P}-points}, Canad. Math. Bull.
  \textbf{62} (2019), no.~4, 856--868.

\bibitem{smallu}
Dilip Raghavan and Saharon Shelah, \emph{A small ultrafilter number at smaller
  cardinals}, Arch. Math. Logic \textbf{59} (2020), no.~3-4, 325--334.

\bibitem{arXiv:2408.04446}
Dilip Raghavan and Juris Stepr{\=a}ns, \emph{Adding ultrafilters to {S}helah's
  model for no nowhere dense ultrafilters}, arXiv:2408.04446, preprint (2024),
  17 pages.

\bibitem{suslincombdic}
Dilip Raghavan and Stevo Todorcevic, \emph{Combinatorial dichotomies and
  cardinal invariants}, Math. Res. Lett. \textbf{21} (2014), no.~2, 379--401.

\bibitem{cofinal}
S.~Todorcevic, \emph{Directed sets and cofinal types}, Trans. Amer. Math. Soc.
  \textbf{290} (1985), no.~2, 711--723.

\end{thebibliography}
\end{document}